\documentclass[11pt]{amsart}
\usepackage{amssymb}
\usepackage[margin=1.07in]{geometry}

\newtheorem{thm}{Theorem}[section]
\newtheorem{lem}[thm]{Lemma}
\newtheorem{prop}[thm]{Proposition}
\newtheorem{cor}[thm]{Corollary}
\theoremstyle{definition}

\theoremstyle{remark}
\newtheorem{remark}[thm]{Remark}

\numberwithin{equation}{section}

\begin{document}

\title[Higher Schwarzians of elliptic double covers]{Higher Schwarzians of Elliptic Double Covers and Eisenstein--Kronecker Functions}
\author{Hicham Saber}
\author{Abdellah Sebbar}
\address{Department of Mathematics, University of Ha'il, Saudi Arabia}
\address{Department of Mathematics and Statistics, University of Ottawa, Ottawa, Ontario K1N 6N5, Canada}
\email{hi.saber@uoh.edu.sa}
\email{asebbar@uottawa.ca}


\begin{abstract}
We determine the Aharonov invariants of order at least two of every
elliptic double cover by evaluating its projective kernel. In orders
at least three, the formula separates the invariant into a constant
Eisenstein term and an Eisenstein--Kronecker function evaluated under
multiplication by two.
The formula determines the ramification principal parts, torsion
specializations, and isogeny traces, including the correction from
two-torsion in an isogeny kernel. It also realizes the same intrinsic
de Rham tensor through every degree-two projection: in higher orders
this is the image of an Eisenstein section, while order two gives the
classical Weierstrass complement to the Hodge line. A separate
differential calculation shows that the Bernoulli-normalized invariant
of order $p-1$ reduces to the quotient of the $p$th and first iterates
of the defining derivation. For an elliptic invariant derivation this
quotient is the Hasse invariant, independently of the rational function
on its separable locus. The universal double cover provides a global
specialization, integral away from two-torsion, whose reduction extends
regularly across that locus. Classical Eisenstein zero theorems and
the supersingular divisor congruence then describe the exactness
loci and their reduction.
\end{abstract}

\maketitle

\section{Introduction}

Aharonov's higher Schwarzians encode the local projective geometry of a
meromorphic function. For an elliptic double cover, the covering
involution makes this entire sequence explicitly accessible. We
calculate its projective kernel and show that the higher invariants
are classical elliptic functions evaluated under multiplication by
two, with a constant Eisenstein contribution. A second result concerns
reduction in characteristic $p$: after Bernoulli normalization, the
order-$p-1$ invariant detects the restricted power of the defining
derivation and, on an elliptic curve, the Hasse invariant.

We use the normalization
\[
\frac{f'(z)}{f(z+t)-f(z)}
=\frac1t-\sum_{n\geq1}S_n[f](z)t^{n-1}.
\]
For $n\geq2$, these invariants are unchanged by linear fractional
postcomposition, and $6S_2[f]$ is the ordinary Schwarzian derivative
\cite{aharonov}. Related higher Schwarzian operators were studied by
Tamanoi \cite{tamanoi}. The projective kernel and its Aharonov diagonal
expansion are already part of the general formalism in
\cite[Eq.~(2.3) and Proposition~4.1]{saber-sebbar}; that paper also
relates the kernel to Faber--Grunsky coefficients and Bernoulli-normalized
cusp moments. Here the essential analytic step is its elliptic
evaluation:
\[
\mathcal P_{\wp}(w,z)
=\wp(w-z)-\wp(w+z)-\frac1{(w-z)^2}.
\]
For a lattice $\Omega$, Theorem~\ref{thm:closed-formula} gives
\begin{equation}\label{eq:intro-closed}
S_2[\wp](z)=-\wp(2z),\qquad
S_n[\wp](z)=G_n(\Omega)-\frac{\wp^{(n-2)}(2z)}{(n-1)!},
\quad n\geq3,
\end{equation}
where $G_n(\Omega)=\sum_{\lambda\in\Omega\setminus\{0\}}\lambda^{-n}$.
Every degree-two map has the form $M(\wp(z-Q))$; the possible choices of
$Q$ differ by two-torsion and give the same higher invariants. Thus
\eqref{eq:intro-closed} determines every elliptic double cover, including
the exact principal part at each ramification point.

The de Rham interpretation records precisely which part of this
formula survives passage to cohomology. For a nonzero invariant
differential $\omega$, Theorem~\ref{thm:intrinsic-class} shows that
\[
\Xi_n(E)=[S_n^\omega[f]\,\omega]\otimes\omega^{\otimes(n-1)}
\]
is independent of the degree-two map $f$ and of the frame $\omega$.
For $n\geq3$, it is the image of
$G_n(\Omega_\omega)\omega^{\otimes n}$ under the Hodge inclusion.
In order two it is
\[
\Xi_2(E)=-[x\omega]\otimes\omega,
\]
where $x$ is the normalized Weierstrass coordinate determined by
$\omega$. This is the classical complementary tensor in Weierstrass
de Rham cohomology, as described in
\cite[Appendix~A1.2, Eqs.~(A1.2.4)--(A1.2.6)]{katz}.
The higher-Schwarzian formula realizes this tensor through every
degree-two projection and gives explicit exact differentials relating
the representatives. Consequently order two is never exact, odd
orders at least three are always exact, and even orders at least four
are exact precisely on the corresponding Eisenstein zero locus.

The arithmetic result begins with an arbitrary derivation. Let
$\partial$ be a derivation of a $p$-torsion-free $\mathbb Z_{(p)}$-algebra,
where $p\geq5$, and suppose that $\partial f$ is a unit. Define the
Aharonov invariants using the Taylor series of the iterates of
$\partial$. Proposition~\ref{prop:Bernoulli-reduction} proves integrality
of the normalized expression and the congruence
\begin{equation}\label{eq:intro-general-reduction}
-\frac{(p-1)!}{B_{p-1}}S_{p-1}^\partial[f]
\equiv\frac{\bar\partial^{\,p}\bar f}{\bar\partial\bar f}\pmod p.
\end{equation}
Only the $p$th derivative introduces a factorial denominator divisible
by $p$ at this order. On a smooth curve over a perfect field of
characteristic $p$, the quotient on the right is independent of $f$
wherever its denominator is nonzero, because the space of rational
derivations has dimension one. For an elliptic
invariant derivation, the classical identity
$\bar\partial^{\,p}=A(E,\bar\omega)\bar\partial$
\cite[\S2.0]{katz} identifies the quotient with the Hasse invariant.
Theorem~\ref{thm:general-Hasse} therefore applies to arbitrary rational
functions on their separable locus, not only to double covers.

The universal double cover is a global corollary of this differential
result. On the framed Weierstrass family its complement of two-torsion
is the principal open set where $y=\partial x$ is invertible.
Corollary~\ref{cor:Hasse-reduction} gives a regular section of the
pullback of the $(p-1)$st Hodge power on this open set; its reduction
extends to the whole universal elliptic curve as the pullback of the
Hasse invariant. The explicit elliptic formula also separates this
section into $E_{p-1}$ and a correction divisible by $p$. Regularity
across two-torsion is asserted only after reduction.

The remaining results express the consequences of the elliptic
identification within classical theory. The nonconstant terms are
Eisenstein--Kronecker functions; their distribution relations yield
isogeny traces, and their torsion values span the Eisenstein spaces of
principal congruence groups. We specify the analytic normalization of
the Kronecker expression used for the generating series and distinguish
it from a meromorphic Poincar\'e-bundle section. The results of
Rankin--Swinnerton-Dyer, Nozaki, and Kohnen describe the resulting
exactness parameters \cite{rankin-swinnerton-dyer,nozaki,kohnen-trans}.
The supersingular polynomial congruence is the classical Eisenstein
case of Kaneko--Zagier \cite[Theorem~1]{kz-ss}.

Sections~\ref{sec:preliminaries}--\ref{sec:double-covers} establish the
elliptic formula and its geometric interpretation.
Sections~\ref{sec:torsion} and \ref{sec:isogenies} treat torsion and
isogenies. Section~\ref{sec:derham} identifies the de Rham classes,
and Section~\ref{sec:hasse-reduction} proves the differential reduction
formula and its elliptic specializations. The classical divisor
applications are collected in Section~\ref{sec:applications}.

\section{Higher Schwarzians and elliptic preliminaries}\label{sec:preliminaries}

Let $D$ be a domain in the Riemann sphere and let $f$ be meromorphic on $D$. At a point where $f$ is holomorphic and $f'(z)\neq0$, define
\begin{equation}\label{eq:ahar-expansion}
\frac{f'(z)}{f(z+t)-f(z)}
=
\frac1t-\sum_{n=1}^{\infty}S_n[f](z)t^{n-1}.
\end{equation}
The first two coefficients are
\begin{equation}\label{eq:first-two}
S_1[f]=\frac{f''}{2f'},
\quad
6S_2[f]=\{f,z\}.
\end{equation}
For $n\geq2$, $S_n[f]$ is invariant under postcomposition of $f$ by an element of $\operatorname{PGL}_2(\mathbb C)$ \cite{aharonov}. Differentiating \eqref{eq:ahar-expansion} gives the universal recurrence
\begin{equation}\label{eq:ahar-recurrence}
(n+1)S_n[f]
=
S_{n-1}'[f]
+\sum_{r=2}^{n-2}S_r[f]S_{n-r}[f],
\quad n\geq3,
\end{equation}
where the sum is empty for $n=3$.

When $E$ is a complex elliptic curve, $f:E\to\mathbb P^1$ is meromorphic, and $\omega$ is a nonzero invariant differential, we write $S_n^\omega[f]$ for the invariant computed in a uniformizing coordinate $z$ with $dz=\omega$. If $\Omega_\omega$ is the corresponding period lattice and $\mathcal H_E=H^0(E,\Omega_E^1)$, then a change of frame gives
\begin{equation}\label{eq:scaling-Sn}
S_n^{c\omega}[f]=c^{-n}S_n^\omega[f],\qquad c\in\mathbb C^\times.
\end{equation}
This notation will be used when the higher Schwarzians are viewed intrinsically on $E$.

The regularized two-point projective kernel, with the normalization of \cite[Eq.~(2.3)]{saber-sebbar}, is
\[
\mathcal P_f(w,z)
=
\frac{f'(w)f'(z)}{(f(w)-f(z))^2}-\frac1{(w-z)^2}.
\]
Differentiating \eqref{eq:ahar-expansion} with respect to $t$ gives its diagonal expansion:
\[
\mathcal P_f(z+t,z)
=
\sum_{n=2}^{\infty}(n-1)S_n[f](z)t^{n-2}.
\]
At a point of local univalence, the regularized diagonal value is therefore
$\mathcal P_f(z,z)=S_2[f](z)=\{f,z\}/6$. Thus the diagonal coefficients recover all the higher invariants of order at least two, with the normalization fixed by \eqref{eq:first-two}.

Let
\[
\Omega=\omega_1\mathbb Z+\omega_2\mathbb Z,
\quad
\frac{\omega_2}{\omega_1}\in\mathbb H,
\]
where $\mathbb H$ is the upper half-plane, and write $\wp$, $\zeta$, $\sigma$, $g_2$ and $g_3$ for the associated Weierstrass functions and invariants. We use
\begin{equation}\label{eq:wp-equation}
(\wp')^2=4\wp^3-g_2\wp-g_3,
\quad
\zeta'=-\wp,
\end{equation}
and
\[
g_2=60G_4(\Omega),
\quad
g_3=140G_6(\Omega).
\]
For $n\geq3$ set
\begin{equation}\label{eq:Gn}
G_n(\Omega)
=
\sum_{\lambda\in\Omega\setminus\{0\}}\lambda^{-n}.
\end{equation}
The sum vanishes for odd $n$. The Laurent expansion of $\zeta$ is
\begin{equation}\label{eq:zeta-Laurent}
\zeta(t)
=
\frac1t-\sum_{m\geq2}G_{2m}(\Omega)t^{2m-1}.
\end{equation}

\section{Solution of the higher-Schwarzian hierarchy}\label{sec:elliptic-formula}

The following identity is the starting point for the explicit calculation.

\begin{prop}\label{prop:key-zeta}
For $z,t$ away from the poles and zeros of the denominators,
\begin{equation}\label{eq:key-zeta}
\frac{\wp'(z)}{\wp(z+t)-\wp(z)}
=
\zeta(t)-\zeta(2z+t)+2\zeta(z).
\end{equation}
\end{prop}

\begin{proof}
The classical zeta addition formula, in the normalization used here, gives
(see, for example, \cite{lang,weil})
\begin{equation}\label{eq:zeta-add1}
\zeta(t)-\zeta(z+t)+\zeta(z)
=
\frac12\frac{\wp'(z+t)+\wp'(z)}{\wp(z+t)-\wp(z)}.
\end{equation}
Applying the same formula to $z+(z+t)$ gives
\begin{equation}\label{eq:zeta-add2}
\zeta(2z+t)-\zeta(z+t)-\zeta(z)
=
\frac12\frac{\wp'(z+t)-\wp'(z)}{\wp(z+t)-\wp(z)}.
\end{equation}
Subtracting \eqref{eq:zeta-add2} from \eqref{eq:zeta-add1} yields \eqref{eq:key-zeta}.
\end{proof}

Differentiating \eqref{eq:key-zeta} with respect to $t$ and setting $w=z+t$ gives the elliptic specialization of the projective kernel:
\[
\mathcal P_{\wp}(w,z)
=
\wp(w-z)-\wp(w+z)-\frac1{(w-z)^2}.
\]
This is an identity of meromorphic functions in the uniformizing variables. For $w=z+t$, it separates the lattice term $\wp(t)-t^{-2}$ from the elliptic term $-\wp(2z+t)$. Their diagonal coefficients give the constant Eisenstein contribution and the pullback by multiplication by two in the following theorem.

\begin{thm}\label{thm:closed-formula}
Let $S_n(z)=S_n[\wp](z)$. Then
\begin{align}
S_1(z)&=\zeta(2z)-2\zeta(z),\label{eq:S1}\\
S_2(z)&=-\wp(2z),\label{eq:S2}\\
S_n(z)&=G_n(\Omega)-\frac{1}{(n-1)!}\wp^{(n-2)}(2z),
\quad n\geq3.\label{eq:closed-formula}
\end{align}
Here $G_n(\Omega)=0$ for odd $n$.
\end{thm}

\begin{proof}
By definition,
\[
\frac{\wp'(z)}{\wp(z+t)-\wp(z)}
=
\frac1t-\sum_{n\geq1}S_n(z)t^{n-1}.
\]
Use Proposition~\ref{prop:key-zeta}. The constant term of
\[
\zeta(t)-\zeta(2z+t)+2\zeta(z)
\]
is $2\zeta(z)-\zeta(2z)$, which is $-S_1(z)$. The coefficient of $t$ is $\wp(2z)$, which gives \eqref{eq:S2}. For $n\geq3$, the coefficient of $t^{n-1}$ is
\[
\frac{1}{(n-1)!}\wp^{(n-2)}(2z)-G_n(\Omega),
\]
where the Eisenstein term occurs only for even $n$. Comparison with the coefficient $-S_n(z)$ gives \eqref{eq:closed-formula}.
\end{proof}

The ordinary Schwarzian identity is immediate from \eqref{eq:first-two} and \eqref{eq:S2}:
\[
\{\wp(z),z\}=-6\wp(2z).
\]

The closed formula gives two recurrences which do not appear in the general Aharonov theory.

\begin{cor}\label{cor:differential-recurrence}
For every $n\geq2$,
\begin{equation}\label{eq:differential-recurrence}
S_n'(z)
=
2n\bigl(S_{n+1}(z)-G_{n+1}(\Omega)\bigr),
\end{equation}
where $G_3=0$ and, for $n=2$, the right side is interpreted using \eqref{eq:closed-formula} for $S_3$.
\end{cor}

\begin{proof}
For $n=2$, equation \eqref{eq:S2} gives
\[
S_2'(z)=-2\wp'(2z).
\]
On the other hand, \eqref{eq:closed-formula} with $n=3$ gives
$S_3(z)=-\wp'(2z)/2$, since $G_3=0$. Thus
$S_2'=4(S_3-G_3)$, as required.

Now let $n\geq3$. Differentiating \eqref{eq:closed-formula} and using the
chain rule yields
\[
S_n'(z)=-\frac{2}{(n-1)!}\wp^{(n-1)}(2z).
\]
The same formula with $n$ replaced by $n+1$ says
\[
S_{n+1}(z)-G_{n+1}(\Omega)
=-\frac{1}{n!}\wp^{(n-1)}(2z).
\]
Multiplication by $2n$ proves the identity.
\end{proof}

\begin{cor}\label{cor:algebraic-recurrence}
For every $n\geq4$,
\begin{equation}\label{eq:algebraic-recurrence}
(n-3)S_n
=
2(n-1)G_n(\Omega)
-
\sum_{r=2}^{n-2}S_rS_{n-r}.
\end{equation}
Thus, once the lattice Eisenstein series are fixed, the higher-Schwarzian sequence satisfies a derivative-free quadratic recurrence.
\end{cor}

\begin{proof}
The universal Aharonov recurrence \eqref{eq:ahar-recurrence} reads
\[
(n+1)S_n=S_{n-1}'+\sum_{r=2}^{n-2}S_rS_{n-r}.
\]
Substitute
\[
S_{n-1}'=2(n-1)(S_n-G_n)
\]
from Corollary~\ref{cor:differential-recurrence}. Moving the term
$2(n-1)S_n$ to the left and multiplying by $-1$ gives exactly
\eqref{eq:algebraic-recurrence}.
\end{proof}

\begin{cor}\label{cor:polynomial-structure}
Put $W=\wp(2z)$ and $V=\wp'(2z)$, and assign weights $2,3,4,6$ to $W,V,g_2,g_3$, respectively. For every $m\geq1$,
\[
S_{2m}\in\mathbb Q[g_2,g_3,W],
\quad
S_{2m+1}\in V\,\mathbb Q[g_2,g_3,W].
\]
These expressions are weighted homogeneous of weights $2m$ and $2m+1$. As polynomials in $W$, their leading terms are
\[
S_{2m}=-W^m+\text{terms of lower $W$-degree},
\]
and
\[
S_{2m+1}
=V\left(-\frac12W^{m-1}+\text{terms of lower $W$-degree}\right).
\]
\end{cor}

\begin{proof}
Give $\wp$, $\wp'$, $g_2$ and $g_3$ the respective weights $2,3,4,6$.
The differential equation
\[
(\wp')^2=4\wp^3-g_2\wp-g_3
\]
and its derivative $\wp''=6\wp^2-g_2/2$ show inductively that every even
derivative of $\wp$ is a polynomial in $\wp,g_2,g_3$, whereas every odd
derivative is $\wp'$ times such a polynomial. Since differentiation raises
weight by one, $\wp^{(2m-2)}$ has weight $2m$. Its highest power of $\wp$
is obtained by repeatedly differentiating the leading term, and is
$(2m-1)!\wp^m$. Differentiation then gives
\[
\wp^{(2m-1)}
=\wp'\left(\frac{(2m)!}{2}\wp^{m-1}+\text{terms of lower $\wp$-degree}\right).
\]
Substitution into Theorem~\ref{thm:closed-formula} produces the stated leading
terms. The constant $G_{2m}$ is a polynomial in $g_2,g_3$ of weight $2m$ and
does not affect the highest $W$-degree. The case $m=1$ follows directly from
$S_2=-W$ and $S_3=-V/2$.
\end{proof}

For example,
\begin{align*}
S_2&=-W,\\
S_3&=-\frac12V,\\
S_4&=-W^2+\frac{g_2}{10},\\
S_5&=-\frac12WV,\\
S_6&=-W^3+\frac{3g_2}{20}W+\frac{3g_3}{28},\\
S_7&=V\left(-\frac12W^2+\frac{g_2}{40}\right),\\
S_8&=-W^4+\frac{g_2}{5}W^2+\frac{g_3}{7}W-\frac{g_2^2}{600}.
\end{align*}

The generating series can also be expressed through sigma functions.
On the uniformizing plane, set
\begin{equation}\label{eq:Kronecker-function}
K_\sigma(u,t)=\frac{\sigma(u+t)}{\sigma(u)\sigma(t)},
\qquad
\Phi(u,t)=K_\sigma(u,t)e^{-\zeta(u)t}.
\end{equation}
The sigma quotient $K_\sigma$ is meromorphic in both variables. Its
quasi-periodicity is the familiar analytic form of the Kronecker
construction \cite{weil}. The exponential normalization used for
$\Phi$ gives
\[
\partial_t\log\Phi(u,t)
=\zeta(u+t)-\zeta(t)-\zeta(u).
\]
Adding $S_1(z)$ to \eqref{eq:key-zeta} therefore gives
\begin{equation}\label{eq:Kronecker-generating}
\frac1t-\sum_{n\geq2}S_n[\wp](z)t^{n-1}
=-\partial_t\log\Phi(2z,t).
\end{equation}
Here the logarithmic derivative means $\partial_t\Phi/\Phi$, so no
choice of logarithm is required.

The normalization in \eqref{eq:Kronecker-function} matters. For fixed
$t\notin\Omega$, the expansions $\sigma(u)=u+O(u^5)$ and
$\zeta(u)=u^{-1}+O(u^3)$ give
\[
\Phi(u,t)=u^{-1}e^{-t/u}\bigl(1+O(u)\bigr)
\qquad (u\longrightarrow0).
\]
Thus $\Phi$ has an essential singularity in $u$ at the origin. It is
an analytic, exponentially normalized Kronecker expression on the
punctured uniformizing domain, not a meromorphic section across the
zero divisor of $E\times E$. Its logarithmic derivative, however, is
meromorphic in the uniformizing variables and is the only expression
used in \eqref{eq:Kronecker-generating}. The algebraic Kronecker
section of the Poincar\'e bundle is a distinct meromorphic object;
see \cite{sprang}. No algebraic extension of the exponential factor
is used here.

\section{Geometric descent to elliptic double covers}\label{sec:double-covers}

The appearance of \(2z\) in Theorem~\ref{thm:closed-formula} reflects the
double-cover structure of the Weierstrass map
\[
\wp:E\longrightarrow \mathbb{P}^1,
\]
whose nontrivial covering involution is \(z\mapsto -z\).

\begin{thm}\label{thm:two-torsion-symmetry}
Let $E=\mathbb C/\Omega$ and let $a\in E$. The following are equivalent.
\begin{enumerate}
\item There is $M_a\in\operatorname{PGL}_2(\mathbb C)$ such that
\[
\wp(z+a)=M_a(\wp(z)).
\]
\item Translation by $a$ descends through the double cover $\wp:E\to\mathbb P^1$.
\item $a\in E[2]$.
\end{enumerate}
Consequently, for every $n\geq2$ and every $T\in E[2]$,
\begin{equation}\label{eq:E2-invariance}
S_n[\wp](z+T)=S_n[\wp](z).
\end{equation}
\end{thm}

\begin{proof}
The equivalence of the first two statements is the definition of descent to the quotient. The deck group of $\wp$ is generated by the involution $[-1]$. If translation $t_a$ descends, it must normalize this deck group. Since
\[
t_a\circ[-1]\circ t_{-a}=t_{2a}\circ[-1],
\]
normalization forces $t_{2a}=1$, hence $2a=0$.

Conversely, let $a\in E[2]$. Translation by $a$ commutes with $[-1]$, hence descends to an automorphism of the quotient $E/\{\pm1\}\simeq\mathbb P^1$. This automorphism is linear fractional, proving the first statement. Equation \eqref{eq:E2-invariance} follows from projective invariance of $S_n$, $n\geq2$.
\end{proof}

\begin{cor}\label{cor:descent-degree}
For every $n\geq2$, the elliptic function $S_n[\wp]$ descends through the quotient $E\to E/E[2]$. In the coordinate $u=2z$ on the quotient, the descended function is
\[
s_n(u)=
\begin{cases}
-\wp(u),&n=2,\\
G_n(\Omega)-\dfrac{1}{(n-1)!}\wp^{(n-2)}(u),&n\geq3.
\end{cases}
\]
It has a unique pole, at the origin, of exact order $n$. Consequently $s_n:E/E[2]\to\mathbb P^1$ has degree $n$.
\end{cor}

\begin{proof}
The descent follows from Theorem~\ref{thm:two-torsion-symmetry}, and the displayed formula follows from Theorem~\ref{thm:closed-formula}. Since
\[
\wp^{(n-2)}(u)=(-1)^{n-2}(n-1)!u^{-n}+O(u^{2-n}),
\]
the pole at the origin has exact order $n$, and there are no other poles on the quotient.
\end{proof}

For a nonzero half-period $T$ with $e_T=\wp(T)$, the descended transformation can also be written explicitly as
\[
\wp(z+T)
=e_T+
\frac{(e_T-e_{T'})(e_T-e_{T''})}{\wp(z)-e_T},
\]
where the three nonzero two-torsion points are $T,T',T''$.

The next result removes the choice of Weierstrass coordinate.

\begin{thm}\label{thm:double-cover}
Let $E$ be a complex elliptic curve and let
\[
f:E\longrightarrow\mathbb P^1
\]
be a meromorphic map of degree two. Let $\iota$ be its deck involution. Then there are $Q\in E$, a lattice uniformization $E\simeq\mathbb C/\Omega$, and $M\in\operatorname{PGL}_2(\mathbb C)$ such that
\begin{equation}\label{eq:double-cover-normal-form}
f(z)=M(\wp(z-Q)).
\end{equation}
For every $n\geq2$,
\begin{equation}\label{eq:double-cover-Sn}
S_n[f](z)=S_n[\wp](z-Q).
\end{equation}
The right side is independent of the choice of $Q$ satisfying \eqref{eq:double-cover-normal-form}.
\end{thm}

\begin{proof}
Since \(f\) has degree two, its nontrivial deck transformation
\(\iota\) is an involution and
\[
f\circ\iota=f.
\]
The Riemann--Hurwitz formula gives
\[
2g(E)-2
=
2\bigl(2g(\mathbb P^1)-2\bigr)
+
\sum_{P\in E}(e_P-1).
\]
Since \(g(E)=1\) and \(g(\mathbb P^1)=0\), this becomes
\[
0=-4+\sum_{P\in E}(e_P-1).
\]
Thus \(f\) has four ramification points, each of ramification
index \(2\). These are precisely the fixed points of the deck
involution \(\iota\).

Choose a fixed point \(Q_0\) of \(\iota\). After translating \(Q_0\)
to the origin, the conjugate
\[
\widetilde{\iota}
=
t_{-Q_0}\circ\iota\circ t_{Q_0}
\]
is an automorphism of \(E\) fixing the origin. Hence
\(\widetilde{\iota}\) is a group automorphism of \(E\). Since it is
a nontrivial involution, it must be the inversion map
\[
P\longmapsto -P.
\]
Consequently,
\[
\iota(P)=2Q_0-P.
\]
Writing \(T=2Q_0\), we may therefore express the deck involution as
\[
\iota(P)=T-P.
\]

More generally, let \(Q\in E\) satisfy \(2Q=T\). Then
\[
\iota(Q+P)
=
T-Q-P
=
Q-P,
\]
and hence
\[
f(Q+P)=f(Q-P).
\]
Thus the meromorphic function
\[
P\longmapsto f(Q+P)
\]
is invariant under \(P\mapsto -P\), and therefore factors through
the quotient
\[
E\longrightarrow E/\{\pm1\}\simeq\mathbb P^1.
\]
Under a lattice uniformization \(E\simeq\mathbb C/\Omega\), the
standard degree-two quotient map for this involution is the
Weierstrass function
\[
\wp:\mathbb C/\Omega\longrightarrow\mathbb P^1.
\]
Since both \(P\mapsto f(Q+P)\) and \(\wp(P)\) are degree-two maps
realizing the same quotient, they differ by an automorphism of
\(\mathbb P^1\). Hence there exists
\(M\in\operatorname{PGL}_2(\mathbb C)\) such that
\[
f(Q+P)=M(\wp(P)).
\]
Replacing \(P\) by \(z-Q\) gives
\[
f(z)=M(\wp(z-Q)),
\]
which proves \eqref{eq:double-cover-normal-form}.

The Aharonov invariants are unchanged by postcomposition with a
Möbius transformation. Therefore, for every \(n\geq2\),
\[
S_n[f](z)
=
S_n[\wp(\,\cdot-Q\,)](z)
=
S_n[\wp](z-Q),
\]
which is \eqref{eq:double-cover-Sn}.

Finally, if \(Q'\) is another point satisfying \(2Q'=T\), then
\[
2(Q'-Q)=0,
\]
so \(Q'-Q\in E[2]\). By
Theorem~\ref{thm:two-torsion-symmetry}, the functions
\(S_n[\wp]\) are invariant under translation by points of \(E[2]\).
It follows that
\[
S_n[\wp](z-Q')
=
S_n[\wp](z-Q),
\]
and hence the right-hand side of \eqref{eq:double-cover-Sn} is
independent of the choice of \(Q\).
\end{proof}

\begin{cor}\label{cor:poles-ramification}
Let
\[
f:E\longrightarrow \mathbb P^1
\]
be a meromorphic map of degree two, and let \(R_f\) denote its reduced
ramification divisor. Fix a nonzero invariant differential \(\omega\) on
\(E\), and compute the Aharonov invariants with respect to a uniformizing
coordinate whose differential is \(\omega\).

Then, for every \(n\geq2\),
\[
\operatorname{div}_{\infty}\bigl(S_n^\omega[f]\bigr)=nR_f.
\]
Equivalently, the poles of \(S_n^\omega[f]\) are precisely the four
ramification points of \(f\), and each of them has order \(n\).

More precisely, let \(P\) be a ramification point of \(f\), and let \(w\)
be the local uniformizing coordinate centered at \(P\), normalized by
\[
w(P)=0,
\quad
dw=\omega.
\]
Then
\begin{equation}\label{eq:principal-part}
S_n^\omega[f](P+w)
=
\frac{(-1)^{n-1}}{2^n w^n}
+
O(1).
\end{equation}
Consequently,
\[
\deg \operatorname{div}_{0}\bigl(S_n^\omega[f]\bigr)
=
\deg \operatorname{div}_{\infty}\bigl(S_n^\omega[f]\bigr)
=
4n.
\]
Thus \(S_n^\omega[f]\) has \(4n\) zeros on \(E\), counted with
multiplicity.
\end{cor}

\begin{proof}
By Theorem~\ref{thm:double-cover}, after choosing a lattice
uniformization \(E\simeq\mathbb C/\Omega\), there exist \(Q\in E\) and
\(M\in\operatorname{PGL}_2(\mathbb C)\) such that
\[
f(z)=M\bigl(\wp(z-Q)\bigr).
\]
Projective invariance of the Aharonov invariants therefore gives
\[
S_n^\omega[f](z)=S_n[\wp](z-Q).
\]

The ramification points of the Weierstrass double cover are precisely
the points of \(E[2]\). Hence the ramification points of \(f\) are the
points \(P\in E\) satisfying
\[
P-Q\in E[2].
\]
In particular,
\[
2(P-Q)\in\Omega.
\]

Fix such a ramification point \(P\), and write \(z=P+w\), where \(w\)
is the local coordinate normalized by \(dw=\omega\). Then
\[
2(z-Q)=2(P-Q)+2w.
\]
Since \(2(P-Q)\) is a lattice point and the Weierstrass function and
all of its derivatives are \(\Omega\)-periodic, we obtain
\[
\wp^{(n-2)}\bigl(2(z-Q)\bigr)
=
\wp^{(n-2)}(2w).
\]

For \(n\geq3\), Theorem~\ref{thm:closed-formula} gives
\[
S_n^\omega[f](P+w)
=
G_n(\Omega)
-
\frac{1}{(n-1)!}\wp^{(n-2)}(2w).
\]
Using the Laurent expansion
\[
\wp^{(n-2)}(u)
=
(-1)^{n-2}(n-1)!u^{-n}
+
O(1),
\]
we find
\[
S_n^\omega[f](P+w)
=
\frac{(-1)^{n-1}}{2^n w^n}
+
O(1).
\]
For \(n=2\), the same conclusion follows directly from
\[
S_2[\wp](z)=-\wp(2z)
\]
and the expansion
\[
\wp(2w)=\frac{1}{4w^2}+O(w^2).
\]
Thus \eqref{eq:principal-part} holds for every \(n\geq2\).

The coefficient of \(w^{-n}\) in \eqref{eq:principal-part} is nonzero.
Hence every ramification point is a pole of exact order \(n\).
Conversely, the formulas above show that no other poles occur. Therefore
\[
\operatorname{div}_{\infty}\bigl(S_n^\omega[f]\bigr)=nR_f.
\]

Finally, a degree-two map from an elliptic curve to \(\mathbb P^1\)
has four ramification points, by the Riemann--Hurwitz formula. Hence
\[
\deg(nR_f)=4n.
\]
Since a nonconstant meromorphic function on a compact Riemann surface
has zero and pole divisors of the same degree, \(S_n^\omega[f]\) has
\(4n\) zeros counted with multiplicity.
\end{proof}

\begin{remark}
The divisor statement is independent of the choice of $\omega$. The numerical
coefficient in \eqref{eq:principal-part} records the normalization $dw=\omega$.
Under $\omega\mapsto c\omega$, the local coordinate becomes $cw$ and the
coefficient transforms with weight $n$.
\end{remark}

\section{Level structures from torsion specialization}\label{sec:torsion}

For $n\geq3$ define
\begin{equation}\label{eq:phi-n}
\phi_n(u;\Omega)
=
\sum_{\lambda\in\Omega}(u+\lambda)^{-n}.
\end{equation}
The series converges absolutely away from the lattice and defines a meromorphic elliptic function of $u$, with poles at the lattice points. In the standard terminology it is the holomorphic member of the classical Eisenstein--Kronecker system; see \cite{weil,sprang}.

\begin{thm}\label{thm:EK-identification}
For every $n\geq3$,
\begin{equation}\label{eq:EK-identification}
S_n[\wp](z)-G_n(\Omega)
=
(-1)^{n-1}\phi_n(2z;\Omega).
\end{equation}
Equivalently, the generating series formed by the projective invariants
$S_n$, $n\geq2$, is the logarithmic derivative of the Kronecker function as
in \eqref{eq:Kronecker-generating}.
\end{thm}

\begin{proof}
For \(n\geq3\), differentiating \(n-2\) times with
respect to \(u\) the expression
\[
\wp(u)
=
\frac{1}{u^2}
+
\sum_{\lambda\in\Omega\setminus\{0\}}
\left(
\frac{1}{(u+\lambda)^2}-\frac{1}{\lambda^2}
\right),
\]
we obtain
\[
\wp^{(n-2)}(u)
=
(-1)^{n-2}(n-1)!
\sum_{\lambda\in\Omega}(u+\lambda)^{-n}.
\]
By the definition of \(\phi_n\), this is
\[
\wp^{(n-2)}(u)
=
(-1)^{n-2}(n-1)!\phi_n(u;\Omega).
\]

Now Theorem~\ref{thm:closed-formula} gives
\[
S_n[\wp](z)
=
G_n(\Omega)
-
\frac{1}{(n-1)!}\wp^{(n-2)}(2z).
\]
Substituting the preceding identity with \(u=2z\), we find
\begin{align*}
S_n[\wp](z)-G_n(\Omega)
&=
-\frac{1}{(n-1)!}
   \wp^{(n-2)}(2z)\\
&=
(-1)^{n-1}\phi_n(2z;\Omega),
\end{align*}
which proves \eqref{eq:EK-identification}.

The equivalent generating-series formulation is precisely
\eqref{eq:Kronecker-generating}.
\end{proof}

Now put $L_\tau=\mathbb Z\tau+\mathbb Z$. For the column vector
$a=(a_1,a_2)^t\in\mathbb Q^2$ write
\[
\tau_a=a_1\tau+a_2.
\]
For $a\notin\mathbb Z^2$ and $n\geq3$ define
\begin{equation}\label{eq:Ena}
E_{n,a}(\tau)=\phi_n(\tau_a;L_\tau).
\end{equation}
Since \(a\in\mathbb Q^2\), the class of
\(\tau_a=a_1\tau+a_2\) in \(\mathbb C/L_\tau\) is a torsion point.
Thus $E_{n,a}(\tau)$
is the specialization of the holomorphic Eisenstein--Kronecker
function \(\phi_n\) at a torsion point of the elliptic curve
\(\mathbb C/L_\tau\); see, for example, \cite{weil}.
If $\ell(a)$ denotes the least positive integer $N$ such that $Na\in\mathbb Z^2$, then $N$ is the exact level of $a$. Equivalently, writing
\[
a\equiv \left(\frac rN,\frac sN\right)^t \pmod{\mathbb Z^2},
\]
the condition $\ell(a)=N$ is equivalent to $\gcd(r,s,N)=1$. Thus exact-level classes are represented by primitive residue vectors modulo $N$, the standard data indexing the cusps of $\Gamma(N)$, up to the usual sign identification of a cusp; see \cite{schoeneberg}.

The modular transformation law is
\begin{equation}\label{eq:Ena-transform}
E_{n,a}(\gamma\tau)(C\tau+D)^{-n}
=
E_{n,\gamma^t a}(\tau),
\quad
\gamma=\begin{pmatrix}A&B\\ C&D\end{pmatrix}\in\operatorname{SL}_2(\mathbb Z),
\end{equation}
 see
\cite{schoeneberg}.

\begin{thm}\label{thm:torsion-specialization}
Let $n\geq3$ and $a\in\mathbb Q^2\setminus\mathbb Z^2$. Then
\begin{equation}\label{eq:torsion-specialization}
S_n[\wp_\tau]\left(\frac{\tau_a}{2}\right)-G_n(L_\tau)
=
(-1)^{n-1}E_{n,a}(\tau).
\end{equation}
For every $N>1$, the functions
\begin{equation}\label{eq:eis-span}
\left\{
S_n[\wp_\tau]\left(\frac{\tau_a}{2}\right)-G_n(L_\tau)
:\ \ell(a)=N
\right\}
\end{equation}
span the Eisenstein subspace of weight $n$ and trivial character for
$\Gamma(N)$, namely the Eisenstein subspace associated with all cusps of
$\Gamma(N)$.
\end{thm}

\begin{proof}
Equation \eqref{eq:torsion-specialization} is Theorem~\ref{thm:EK-identification} evaluated at $z=\tau_a/2$. The spanning assertion is the classical theorem that the $E_{n,a}$ of exact level $N$ span the Eisenstein subspace of weight $n$ for $\Gamma(N)$; see \cite{schoeneberg} and \cite[Theorem~2.1]{pasol}.
The normalization in \eqref{eq:Ena} agrees with those references up to a
nonzero scalar depending only on $n$, so it does not affect the resulting span. Under the description above, exact level $N$ corresponds to the primitive residue vectors indexing the cusps of $\Gamma(N)$. For odd $n$ and $N=2$, both sides of the
spanning assertion are zero: one has $E_{n,-a}=-E_{n,a}$ but $-a=a$ modulo
$\mathbb Z^2$, while $-I\in\Gamma(2)$ forces every modular form of odd weight
and trivial character to vanish.
\end{proof}

Hence torsion specializations of the order-\(n\) higher Schwarzian recover the Eisenstein subspace at every principal congruence level.

The logarithmic derivative of the Kronecker function and its expansion in
Eisenstein--Kronecker functions are classical; see, for example,
Weil~\cite{weil}. Theorem~\ref{thm:EK-identification} and
Equation~\eqref{eq:Kronecker-generating} identify this classical generating
series with the full sequence of Aharonov invariants associated with the
elliptic double cover.

The two-variable modular law follows directly from the closed formula.

\begin{prop}\label{prop:modular-two-variable}
Let $n\geq2$ and $\gamma=\left(\begin{smallmatrix}a&b\\c&d\end{smallmatrix}\right)\in\operatorname{SL}_2(\mathbb Z)$. Then
\begin{equation}\label{eq:modular-two-variable}
S_n[\wp_{\gamma\tau}]\left(\frac{z}{c\tau+d}\right)
=
(c\tau+d)^nS_n[\wp_\tau](z).
\end{equation}
Moreover,
\[
S_n[\wp_\tau]\left(z+\frac{\lambda}{2}\right)
=S_n[\wp_\tau](z),
\quad \lambda\in L_\tau.
\]
\end{prop}

\begin{proof}
The lattices satisfy $L_{\gamma\tau}=(c\tau+d)^{-1}L_\tau$ since $\gamma\in \operatorname{SL}_2(\mathbb Z) $. If
$\Lambda'=\alpha\Lambda$, then direct substitution in the defining series
gives
\[
G_n(\Lambda')=\alpha^{-n}G_n(\Lambda),\quad
\wp_{\Lambda'}^{(r)}(\alpha u)=\alpha^{-r-2}\wp_\Lambda^{(r)}(u).
\]
Take $\alpha=(c\tau+d)^{-1}$ and $r=n-2$, and substitute these two scaling
relations into Theorem~\ref{thm:closed-formula}. Both terms acquire the factor
$(c\tau+d)^n$, proving \eqref{eq:modular-two-variable}. Finally,
$\lambda/2$ represents a point of $E[2]$, so the second assertion follows from
Theorem~\ref{thm:two-torsion-symmetry}.
\end{proof}

\section{Isogeny functoriality and trace identities}\label{sec:isogenies}

For $n\geq3$ set
\begin{equation}\label{eq:Tn}
T_n^\Omega(z)
=
S_n[\wp_\Omega](z)-G_n(\Omega)
=
(-1)^{n-1}\phi_n(2z;\Omega).
\end{equation}

\begin{thm}\label{thm:distribution}
Let $N\geq1$, let $d=(N,2)$, and let $n\geq3$. Then
\begin{equation}\label{eq:distribution}
\sum_{P\in E[N]}T_n^\Omega(z+P)
=
N^n d^{2-n}
T_n^\Omega\left(\frac Nd z\right).
\end{equation}
In particular, if $N$ is odd,
\begin{equation}\label{eq:distribution-odd}
\sum_{P\in E[N]}T_n^\Omega(z+P)
=
N^nT_n^\Omega(Nz).
\end{equation}
\end{thm}

\begin{proof}
Recall that, over \(\mathbb C\),
\[
E[N]\simeq (\mathbb Z/N\mathbb Z)^2;
\]
see, for example, \cite[Ch.~II, \S1]{lang}. Let \(d=(N,2)\).
It follows that multiplication by \(2\) on \(E[N]\) has kernel
\[
\ker([2]\vert_{E[N]})=E[d],
\]
which has order \(d^2\), and its image is \(E[N/d]\). Thus every point of
\(E[N/d]\) has exactly \(d^2\) preimages in \(E[N]\) under multiplication
by \(2\).

Using \eqref{eq:Tn}, we therefore obtain
\begin{align*}
\sum_{P\in E[N]}T_n^\Omega(z+P)
&=
(-1)^{n-1}
\sum_{P\in E[N]}
\phi_n(2z+2P;\Omega)\\
&=
(-1)^{n-1}d^2
\sum_{Q\in E[N/d]}
\phi_n(2z+Q;\Omega).
\end{align*}

We now use the classical distribution relation for the
Eisenstein--Kronecker functions
\begin{equation}\label{eq:phi-distribution}
\sum_{Q\in E[M]}
\phi_n(u+Q;\Omega)
=
M^n\phi_n(Mu;\Omega),
\quad M\geq1.
\end{equation}
See \cite{weil}; more generally, distribution relations for the
Kronecker section and Eisenstein--Kronecker series are treated in
\cite[Appendix~A]{sprang}.

The same identity follows directly from the lattice definition. Choose representatives of
\[
E[M]\simeq M^{-1}\Omega/\Omega.
\]
Since \(n\geq3\), the defining series is absolutely convergent, and hence
we may rearrange it:
\begin{align*}
\sum_{Q\in E[M]}\phi_n(u+Q;\Omega)
&=
\sum_{Q\in M^{-1}\Omega/\Omega}
\sum_{\lambda\in\Omega}
(u+Q+\lambda)^{-n}\\
&=
\sum_{\mu\in M^{-1}\Omega}
(u+\mu)^{-n}\\
&=
M^n
\sum_{\lambda\in\Omega}
(Mu+\lambda)^{-n}\\
&=
M^n\phi_n(Mu;\Omega).
\end{align*}

Taking \(M=N/d\) and \(u=2z\) in
\eqref{eq:phi-distribution} gives
\begin{align*}
\sum_{P\in E[N]}T_n^\Omega(z+P)
&=
(-1)^{n-1}d^2
\left(\frac{N}{d}\right)^n
\phi_n\left(\frac{2N}{d}z;\Omega\right)\\
&=
N^n d^{\,2-n}
T_n^\Omega\left(\frac{N}{d}z\right),
\end{align*}
which proves \eqref{eq:distribution}.

If \(N\) is odd, then \(d=1\), and the formula reduces to
\[
\sum_{P\in E[N]}T_n^\Omega(z+P)
=
N^nT_n^\Omega(Nz),
\]
which is \eqref{eq:distribution-odd}.
\end{proof}

The same argument gives a trace identity for odd-degree isogenies.

\begin{thm}\label{thm:isogeny-trace}
Let $\Omega\subset\Omega'$ be lattices of odd index $D$, and let
\[
\varphi:\mathbb C/\Omega\longrightarrow\mathbb C/\Omega'
\]
be the induced isogeny. For $n\geq3$,
\begin{equation}\label{eq:isogeny-trace}
\sum_{P\in\ker\varphi}T_n^\Omega(z+P)
=
T_n^{\Omega'}(\varphi(z)).
\end{equation}
\end{thm}

\begin{proof}
Since $D$ is odd, multiplication by $2$ permutes $\ker\varphi$. Therefore
\[
\sum_{P\in\ker\varphi}\phi_n(2z+2P;\Omega)
=
\sum_{P\in\ker\varphi}\phi_n(2z+P;\Omega).
\]
The cosets $P+\Omega$ for $P\in\Omega'/\Omega$ partition $\Omega'$. Hence the last sum is $\phi_n(2z;\Omega')$. Multiply by $(-1)^{n-1}$ and use \eqref{eq:Tn}.
\end{proof}

Equivalently, the left side of \eqref{eq:isogeny-trace} is the trace of the
meromorphic function $T_n^\Omega$ along $\varphi$, pulled back to the source,
while the right side is the pullback of $T_n^{\Omega'}$ from the target.

The odd-degree hypothesis in Theorem~\ref{thm:isogeny-trace} is precisely
the condition that multiplication by \(2\) be an automorphism of the kernel.
For an arbitrary finite-index inclusion \(\Omega\subset\Omega'\), set
\[
K=\Omega'/\Omega,\quad
\Omega_2=\Omega+2\Omega',
\quad
h=\#K[2].
\]
Then the image of multiplication by \(2\) on \(K\) is naturally identified
with
\[
2K\simeq \Omega_2/\Omega,
\]
and every element of this image has exactly \(h\) preimages in \(K\).
Thus \(\Omega_2\) records the part of the target lattice obtained after
doubling the elements of the isogeny kernel, while \(h\) measures the
failure of multiplication by \(2\) to be injective on that kernel.
If the isogeny has odd degree, then \(K[2]=0\), so \(h=1\) and
\(\Omega_2=\Omega'\), recovering the situation of
Theorem~\ref{thm:isogeny-trace}.

\begin{prop}\label{prop:isogeny-even}
For an arbitrary inclusion $\Omega\subset\Omega'$ of finite index and
$n\geq3$, one has
\[
\sum_{P\in\Omega'/\Omega}T_n^\Omega(z+P)
=hT_n^{\Omega_2}(z).
\]
Consequently Theorem~\ref{thm:isogeny-trace} is exactly the case in which
multiplication by two is an automorphism of the kernel.
\end{prop}

\begin{proof}
Multiplication by two on $K=\Omega'/\Omega$ has kernel $K[2]$ of order $h$
and image represented by the quotient lattice $\Omega_2/\Omega$. Hence
\begin{align*}
\sum_{P\in K}T_n^\Omega(z+P)
&=(-1)^{n-1}h
  \sum_{Q\in\Omega_2/\Omega}\phi_n(2z+Q;\Omega)\\
&=(-1)^{n-1}h\phi_n(2z;\Omega_2)
=hT_n^{\Omega_2}(z),
\end{align*}
because the cosets $Q+\Omega$ partition $\Omega_2$.
\end{proof}

Thus, in even degree, the trace factors through the isogeny with target
$\mathbb C/\Omega_2$, and the factor $h$ measures the two-torsion lost under
multiplication by two on the original kernel.

\begin{remark}
The distribution relation for the full Eisenstein--Kronecker system has an algebraic formulation through the Poincar\'e bundle; see \cite{sprang}. Theorems~\ref{thm:distribution} and \ref{thm:isogeny-trace} are the holomorphic specialization obtained after the higher-Schwarzian identification \eqref{eq:EK-identification}.
\end{remark}

\section{The de Rham class and the classical Weierstrass complement}
\label{sec:derham}

For $n\geq3$, Theorem~\ref{thm:closed-formula} gives the exact decomposition
\begin{equation}\label{eq:exact-decomposition}
S_n[\wp](z)
=G_n(\Omega)+\frac{d}{dz}
\left(-\frac{\wp^{(n-3)}(2z)}{2(n-1)!}\right).
\end{equation}
The displayed primitive is elliptic. Hence, for $\lambda\in\Omega$
and a path avoiding the poles,
\begin{equation}\label{eq:period-general}
\int_{z_0}^{z_0+\lambda}S_n[\wp](z)\,dz=G_n(\Omega)\lambda.
\end{equation}
For order two, define the quasi-period $\eta(\lambda)$ by
$\zeta(z+\lambda)-\zeta(z)=\eta(\lambda)$. Then
\begin{equation}\label{eq:order-two-periods}
S_2[\wp](z)\,dz=d\!\left(\tfrac12\zeta(2z)\right),
\qquad
\int_{z_0}^{z_0+\lambda}S_2[\wp](z)\,dz=\eta(\lambda).
\end{equation}
Here the primitive is not elliptic.

Corollary~\ref{cor:poles-ramification} shows that all these differentials
have zero residues. We use the standard realization
\cite{hartshorne-derham}
\[
H^1_{\mathrm{dR}}(E)
\simeq
\frac{\{\text{meromorphic differentials of the second kind on }E\}}
{\{dF:F\in\mathbb C(E)\}}.
\]
In particular, \eqref{eq:exact-decomposition} gives
\begin{equation}\label{eq:cohomology-class}
[S_n[\wp](z)\,dz]=G_n(\Omega)[dz],\qquad n\geq3.
\end{equation}
We use the notation $S_n^\omega[f]$, $\Omega_\omega$, and $\mathcal H_E$
introduced in Section~\ref{sec:preliminaries}.

\begin{thm}\label{thm:intrinsic-class}
Let $E$ be a complex elliptic curve, let $f:E\to\mathbb P^1$ have
degree two, and let $\omega$ be a nonzero invariant differential.
For every $n\geq2$, the tensor
\begin{equation}\label{eq:Xi}
\Xi_n(E)
=[S_n^\omega[f]\,\omega]\otimes\omega^{\otimes(n-1)}
\in H^1_{\mathrm{dR}}(E)\otimes\mathcal H_E^{\otimes(n-1)}
\end{equation}
is independent of $f$ and $\omega$. For $n\geq3$,
\begin{equation}\label{eq:Xi-value}
\Xi_n(E)=G_n(\Omega_\omega)[\omega]\otimes\omega^{\otimes(n-1)}.
\end{equation}
In order two, in the uniformization with $dz=\omega$,
\begin{equation}\label{eq:Xi-two}
\Xi_2(E)=[-\wp(z)\,dz]\otimes\omega.
\end{equation}
\end{thm}

\begin{proof}
For $n\geq3$, translating the source in
\eqref{eq:exact-decomposition} changes only the elliptic primitive.
The cohomology class is therefore unchanged. Projective postcomposition
also leaves it unchanged, so Theorem~\ref{thm:double-cover} proves
independence of the degree-two map in these orders.

For order two, the difference caused by a translation $Q$ is
\begin{equation}\label{eq:order-two-translation}
\bigl(-\wp(2z-2Q)+\wp(2z)\bigr)\,dz
=d\!\left(\frac{\zeta(2z-2Q)-\zeta(2z)}2\right).
\end{equation}
The primitive is elliptic in $z$, since its two quasi-period increments
cancel. Thus the order-two class is also independent of the projection.
Moreover, $\tfrac12\zeta(2z)-\zeta(z)$ is elliptic, and its derivative is
$-\wp(2z)+\wp(z)$. This proves \eqref{eq:Xi-two}.

Finally, under $\omega\mapsto c\omega$, the differential
$S_n^\omega[f]\omega$ is multiplied by $c^{1-n}$, whereas
$\omega^{\otimes(n-1)}$ is multiplied by $c^{n-1}$. These factors cancel
for every $n\geq2$. Formula \eqref{eq:Xi-value} follows from
\eqref{eq:cohomology-class}.
\end{proof}

The order-two tensor has a classical description. Let
\[
y^2=4x^3-g_2x-g_3,\qquad \omega=\frac{dx}{y},
\]
be the normalized short Weierstrass model determined by $(E,\omega)$,
and put $\eta_\omega=x\omega$. In the uniformization $dz=\omega$ one
has $x=\wp(z;\Omega_\omega)$, so \eqref{eq:Xi-two} becomes
\begin{equation}\label{eq:Xi-two-classical}
\Xi_2(E)=-[\eta_\omega]\otimes\omega.
\end{equation}
The forms $\omega$ and $\eta_\omega$ give the classical Weierstrass
basis of de Rham cohomology. To see directly why the second class lies
outside the Hodge line, choose an oriented period basis
$\lambda_1,\lambda_2$ and write $\eta_i=\eta(\lambda_i)$. The periods
of $-\eta_\omega$ are $\eta_1,\eta_2$, and the Legendre relation
\begin{equation}\label{eq:Legendre}
\eta_1\lambda_2-\eta_2\lambda_1=2\pi i
\end{equation}
\cite{lang,weil} prevents this class from being a multiple of $[\omega]$.
In particular it is nonzero, and the two classes form a basis.

Replacing $\omega$ by $c\omega$ replaces $x$ by $c^{-2}x$ and
$\eta_\omega$ by $c^{-1}\eta_\omega$. Thus
\begin{equation}\label{eq:Weierstrass-splitting}
\mathcal H_E\oplus\mathcal H_E^{-1}
\xrightarrow{\ \sim\ }H^1_{\mathrm{dR}}(E),
\qquad
(a\omega,b\omega^{-1})\longmapsto a[\omega]+b[\eta_\omega]
\end{equation}
is independent of the frame. This is the classical splitting recorded
by Katz \cite[Appendix~A1.2, Eqs.~(A1.2.4)--(A1.2.6)]{katz}; the same
construction gives a splitting of the underlying vector bundles for
smooth elliptic families where $6$ is invertible. No compatibility
with the Gauss--Manin connection is asserted here.
Equation~\eqref{eq:Xi-two-classical} identifies the order-two
higher-Schwarzian tensor with the negative of this classical
complementary tensor. Theorem~\ref{thm:intrinsic-class} realizes it
through any double cover, with the exact change of representative
written in \eqref{eq:order-two-translation}.

For $n\geq3$, \eqref{eq:Xi-value} identifies $\Xi_n$ with the image of
$G_n(\Omega_\omega)\omega^{\otimes n}$ under
\[
\mathcal H_E^{\otimes n}\longrightarrow
H^1_{\mathrm{dR}}(E)\otimes\mathcal H_E^{\otimes(n-1)}.
\]
Thus, in higher orders the cohomology retains exactly the Eisenstein
term.

\begin{cor}\label{cor:exactness-criterion}
For a degree-two map $f:E\to\mathbb P^1$, the differential
$S_2^\omega[f]\omega$ is never exact. For $n\geq3$,
$S_n^\omega[f]\omega$ is exact if and only if
$G_n(\Omega_\omega)=0$. In particular every odd order at least three
is exact, and exactness in an even order at least four depends only
on the elliptic curve.
\end{cor}

\begin{proof}
For order two, the Legendre relation \eqref{eq:Legendre} excludes the
zero class. For $n\geq3$, \eqref{eq:exact-decomposition} proves
sufficiency. Conversely, exactness forces all periods to vanish,
and \eqref{eq:period-general} applied to a nonzero lattice period
gives $G_n(\Omega_\omega)=0$. The assertion for odd orders follows by
pairing $\lambda$ with $-\lambda$ in the lattice sum. Projection
independence follows from Theorem~\ref{thm:intrinsic-class}.
\end{proof}

\section{Bernoulli reduction and the Hasse invariant}
\label{sec:hasse-reduction}

The reduction in order $p-1$ has a differential origin independent of
the degree-two hypothesis. We prove the local identity first and then
specialize it to invariant derivations on elliptic curves. The
universal double cover is a global corollary. Throughout this section
$p\geq5$ is prime and the Bernoulli numbers are normalized by $B_2=1/6$.

\subsection{A general differential reduction formula}

Let $B$ be a $p$-torsion-free $\mathbb Z_{(p)}$-algebra, let $\partial$
be a $\mathbb Z_{(p)}$-derivation of $B$, and let $f\in B$ with
$\partial f\in B^\times$. In $B[1/p]((t))$, define
\begin{equation}\label{eq:differential-Aharonov}
\frac{\partial f}{\displaystyle\sum_{j\geq1}
 (\partial^j f)t^j/j!}
=\frac1t-\sum_{n\geq1}S_n^\partial[f]t^{n-1}.
\end{equation}
This is the usual Aharonov expansion when $\partial=d/dz$. A bar
will denote reduction modulo $p$.

\begin{prop}\label{prop:Bernoulli-reduction}
Under these hypotheses, the normalized expression
\[
\mathcal S_{p-1}^\partial[f]
:=-\frac{(p-1)!}{B_{p-1}}S_{p-1}^\partial[f]
\]
belongs to $B$, and
\begin{equation}\label{eq:universal-Bernoulli-reduction}
\overline{\mathcal S_{p-1}^\partial[f]}
=\frac{\bar\partial^{\,p}\bar f}{\bar\partial\bar f}
\quad\text{in } B/pB.
\end{equation}
More precisely,
\begin{equation}\label{eq:leading-p-derivative}
S_{p-1}^\partial[f]
=\frac{\partial^p f}{p!\,\partial f}+R_p[f],
\qquad R_p[f]\in B.
\end{equation}
\end{prop}

\begin{proof}
Put
\[
b_j=\frac{\partial^{j+1}f}{(j+1)!\,\partial f}.
\]
The left side of \eqref{eq:differential-Aharonov} is
$t^{-1}(1+\sum_{j\geq1}b_jt^j)^{-1}$. In the coefficient of $t^{p-1}$
of the inverse series, the only term containing $b_{p-1}$ is
$-b_{p-1}$; every other term is an integral polynomial in
$b_1,\ldots,b_{p-2}$. For $j\leq p-2$, the denominator $(j+1)!$ is a
unit in $\mathbb Z_{(p)}$. Comparing this coefficient with
$-S_{p-1}^\partial[f]$ proves \eqref{eq:leading-p-derivative}.

The von Staudt--Clausen theorem gives
\begin{equation}\label{eq:Bernoulli-p-unit}
pB_{p-1}\in\mathbb Z_{(p)}^\times,
\qquad pB_{p-1}\equiv-1\pmod p
\end{equation}
\cite[pp.~233--237]{ireland-rosen}. Multiplying
\eqref{eq:leading-p-derivative} by $-(p-1)!/B_{p-1}$ gives
\[
\mathcal S_{p-1}^\partial[f]
=-\frac1{pB_{p-1}}\frac{\partial^p f}{\partial f}
 -\frac{(p-1)!}{B_{p-1}}R_p[f].
\]
Both terms belong to $B$. The coefficient of $R_p[f]$ lies in
$p\mathbb Z_{(p)}$, while $-1/(pB_{p-1})$ reduces to $1$. This proves
\eqref{eq:universal-Bernoulli-reduction}.
\end{proof}

This proof shows that the normalized expression is a universal
$\mathbb Z_{(p)}$-polynomial in the derivatives through order $p$, with
powers of $\partial f$ inverted. Thus its reduction depends only on
the reduced differential algebra and not on an integral lift. It is
this integral expression which is reduced; no Taylor expansion with
$1/p!$ as a coefficient is asserted in characteristic $p$.

\begin{remark}\label{rem:one-dimensional-reduction}
Map independence already holds for rational derivations on a curve.
Let $C$ be a smooth geometrically integral curve over a perfect field
$k$ of characteristic $p$, put $K=k(C)$, and let
$D\in\operatorname{Der}_k(K)$ be nonzero. Since
$\Omega^1_{K/k}$ has dimension one over $K$, so does
$\operatorname{Der}_k(K)$. The iterated Leibniz rule makes $D^p$ a
$k$-derivation, because all the intermediate binomial coefficients
vanish in characteristic $p$. Consequently there is a unique
$a_D\in K$ such that
\[
D^p=a_DD,\qquad \frac{D^pf}{Df}=a_D\quad\text{whenever }Df\neq0.
\]
The characteristic-$p$ evaluation of the universal polynomial in
Proposition~\ref{prop:Bernoulli-reduction} is therefore $a_D$ for
every such $f$. In general $a_D$ is a rational function on $C$.
For an elliptic invariant derivation, the coefficient is pulled back
from the base and is identified below with the Hasse invariant.
\end{remark}

\subsection{Independence of the rational function on an elliptic curve}

Let $E$ be an elliptic curve over an $\mathbb F_p$-algebra $R$, let
$\omega$ generate its invariant differentials, and let $\partial$ be
the invariant derivation dual to $\omega$. Write
\[
A(E)=a_\omega\,\omega^{\otimes(p-1)}
\]
for the Hasse invariant. Thus $a_\omega=A(E,\omega)$ is its coefficient
in this frame.

\begin{lem}[Restricted power of the invariant derivation]
\label{lem:restricted-Hasse}
One has
\begin{equation}\label{eq:restricted-Hasse}
\partial^p=a_\omega\partial.
\end{equation}
\end{lem}

\begin{proof}
In characteristic $p$, the iterated Leibniz rule shows that $\partial^p$
is a derivation, since $\binom pj=0$ for $0<j<p$. It is again
translation-invariant and hence is a scalar multiple of $\partial$.
To identify the scalar, use the canonical self-duality of an elliptic
curve to identify $H^1(E,\mathcal O_E)$ with the invariant tangent
space. Under this identification, Frobenius on $H^1(E,\mathcal O_E)$
acts by taking the $p$th iterate of an invariant derivation. Its
coefficient in the basis dual to $\omega$ is the definition of
$a_\omega$. This is the standard description in
\cite[\S2.0, pp.~97--98]{katz}, and proves
\eqref{eq:restricted-Hasse} with the stated normalization.
\end{proof}

\begin{thm}[Map-independent Hasse reduction]\label{thm:general-Hasse}
Let $R$ be a $p$-torsion-free $\mathbb Z_{(p)}$-algebra, let $E/R$ be
an elliptic scheme, and choose a frame $\omega$ of its invariant
differentials. Let $U\subset E$ be an affine open and $f\in\mathcal O(U)$
with $\partial f\in\mathcal O(U)^\times$, where $\partial$ is dual to
$\omega$. Then
\[
\mathcal S_{p-1}^\omega[f]
:=-\frac{(p-1)!}{B_{p-1}}S_{p-1}^\partial[f]
\in\mathcal O(U),
\]
and, on $\bar U$,
\begin{equation}\label{eq:general-Hasse-reduction}
\overline{\mathcal S_{p-1}^\omega[f]}=a_{\bar\omega}.
\end{equation}
Consequently the reduction of
$\mathcal S_{p-1}^\omega[f]\omega^{\otimes(p-1)}$
is the restriction of the pullback of $A$. On every geometric special
fiber meeting $\bar U$, it extends uniquely to the entire fiber as
that regular section. The extension is independent of $f$ and of its
integral lift, and vanishes identically precisely when the fiber is
supersingular.
\end{thm}

\begin{proof}
Smoothness of $E/R$ and $p$-torsion-freeness of $R$ imply
$p$-torsion-freeness of $\mathcal O(U)$. The invariant derivation
preserves this ring. Apply Proposition~\ref{prop:Bernoulli-reduction}
and then Lemma~\ref{lem:restricted-Hasse} to obtain
\[
\overline{\mathcal S_{p-1}^\omega[f]}
=\frac{\bar\partial^{\,p}\bar f}{\bar\partial\bar f}
=a_{\bar\omega}.
\]
The denominator is a unit by hypothesis. Multiplying by
$\bar\omega^{\otimes(p-1)}$ gives the asserted equality of sections.
A nonempty open subset of a geometric elliptic curve is dense, so this
equality determines its regular extension uniquely. The vanishing
criterion is the usual one for the Hasse invariant
\cite[\S2.0]{katz}.

For a unit $c$ pulled back from the base, replacing $\omega$ by
$c\omega$ replaces $\partial$ by $c^{-1}\partial$ and
$S_{p-1}^\partial[f]$ by $c^{-(p-1)}S_{p-1}^\partial[f]$.
The tensor therefore glues under changes of frame. Its reduction is
independent of the map and its lift because it is the pullback of $A$.
\end{proof}

For a nonconstant separable rational function on a geometric special
fiber, the condition $\bar\partial\bar f\neq0$ holds on a dense open
set, after removing its poles and ramification points. The theorem
applies on such an open whenever an integral lift is chosen there;
equivalently, the universal polynomial of
Proposition~\ref{prop:Bernoulli-reduction} may be evaluated directly on
the characteristic-$p$ derivatives. The assertion does not include a
function with identically zero reduced derivative. In that case the
required inversion is unavailable. The characteristic-$p$ result thus
concerns the separable locus and does not impose a degree-two
hypothesis.

\subsection{The universal double cover}

Over the framed Weierstrass base
\[
R_p=\mathbb Z_{(p)}[E_4,E_6,\Delta^{-1}],
\qquad \Delta=\frac{E_4^3-E_6^2}{1728},
\]
let $\mathcal E$ be the smooth projective completion of
\begin{equation}\label{eq:universal-algebraic-curve}
y^2=4x^3-\frac{E_4}{12}x+\frac{E_6}{216},
\qquad \omega=\frac{dx}{y}.
\end{equation}
Write $e$ for the zero section. The scaling
\[
(E_4,E_6,x,y)\longmapsto(c^4E_4,c^6E_6,c^2x,c^3y)
\]
sends $\omega$ to $c^{-1}\omega$; its quotient is the moduli stack
$\mathcal M_{1,1}$ over $\mathbb Z_{(p)}$
\cite[\S1.1 and Appendix~A1]{katz}. The coordinate $x$ gives the
canonical quotient by inversion, up to a projective change of the
target coordinate. The invariant derivation satisfies
\begin{equation}\label{eq:algebraic-derivation}
\partial x=y,\qquad
\partial y=6x^2-\frac{E_4}{24},\qquad
\partial E_4=\partial E_6=0.
\end{equation}

Put $\mathcal U=\mathcal E\setminus\mathcal E[2]$. The nonzero
two-torsion points in the affine chart $\mathcal E\setminus e$ are
cut out by $y=0$. Therefore
\begin{equation}\label{eq:universal-open}
\mathcal U=D(y)=\operatorname{Spec}B_p,\qquad
B_p=R_p[x,y,y^{-1}]\big/
\left(y^2-4x^3+\frac{E_4}{12}x-\frac{E_6}{216}\right).
\end{equation}
The derivation preserves $B_p$, and $\partial x=y$ is a unit there.
Thus this entire open set satisfies the hypotheses of
Theorem~\ref{thm:general-Hasse} with $f=x$.

For even $k\geq4$, write
\begin{equation}\label{eq:normalized-S}
\mathcal S_k^\omega[x]= -\frac{k!}{B_k}S_k^\omega[x],
\qquad
\mathcal S_k=\mathcal S_k^\omega[x]\omega^{\otimes k}.
\end{equation}
For general $k$ these are rational expressions over $\mathbb Q$.
The second expression descends to a rational section on the unframed
universal curve, by \eqref{eq:scaling-Sn} and projective invariance.

\begin{cor}[Universal double-cover specialization]\label{cor:Hasse-reduction}
Let $\pi:\mathcal E\to\mathcal M_{1,1}$ be the universal elliptic
curve over $\mathbb Z_{(p)}$, and let
$\mathcal U=\mathcal E\setminus\mathcal E[2]$.
The section $\mathcal S_{p-1}$ is regular on $\mathcal U$ and takes
values in $\pi^*\omega_{\mathrm{Hdg}}^{\otimes(p-1)}$. Its reduction
satisfies
\begin{equation}\label{eq:Hasse-reduction}
\left.\overline{\mathcal S}_{p-1}\right|_{\bar{\mathcal U}}
=\left.\pi^*A\right|_{\bar{\mathcal U}},
\end{equation}
and extends uniquely to all of $\bar{\mathcal E}$ as $\pi^*A$.
The extension vanishes identically on a geometric fiber precisely
when that fiber is supersingular.
\end{cor}

\begin{proof}
On the framed family, apply Theorem~\ref{thm:general-Hasse} to
$B_p$ and $x$, using \eqref{eq:universal-open}. It gives the
integrality of $\mathcal S_{p-1}^\omega[x]$ and its reduction to
$A(E,\bar\omega)$. Multiplication by $\omega^{\otimes(p-1)}$ and
descent under frame changes give \eqref{eq:Hasse-reduction}.
The pullback $\pi^*A$ is regular on the whole special fiber.
The complement of the finite \'etale two-torsion subgroup is
schematically dense in that fiber: this can be checked on the smooth
framed family, where no irreducible component is contained in the
finite relative divisor. Thus a regular extension is unique.
The vanishing assertion is the Hasse criterion for supersingularity.
\end{proof}

The elliptic formula gives an additional explicit description of this
section. In the complex uniformization, put
\begin{equation}\label{eq:P-D}
x(u,\tau)=(2\pi i)^{-2}\wp(u;L_\tau),\qquad
\partial=(2\pi i)^{-1}\frac{d}{du}.
\end{equation}
Then $y=\partial x$, $\omega=2\pi i\,du$, and for even $k\geq4$,
\begin{equation}\label{eq:G-Bernoulli}
G_k(L_\tau)=-\frac{B_k}{k!}(2\pi i)^kE_k(\tau).
\end{equation}
Substitution in Theorem~\ref{thm:closed-formula} yields the algebraic
identity over $\mathbb Q$
\begin{equation}\label{eq:normalized-S-formula}
\mathcal S_k^\omega[x]
=E_k+\frac{k}{B_k}[2]^*(\partial^{k-2}x).
\end{equation}
The powers of $2\pi i$ only compare the analytic and algebraic frames.
In particular, $S_k^\omega[x]=(2\pi i)^{-k}S_k[\wp_\tau]$.

Equation~\eqref{eq:algebraic-derivation} shows that
\begin{equation}\label{eq:algebraic-integrality}
\partial^r x\in R_p[x,y]/
\left(y^2-4x^3+\frac{E_4}{12}x-\frac{E_6}{216}\right),
\qquad r\geq0.
\end{equation}
Since $[2]^{-1}(e)=\mathcal E[2]$, the map $[2]$ sends
$\mathcal U$ into $\mathcal E\setminus e$. Hence every
$[2]^*(\partial^r x)$ is regular on $\mathcal U$.
For $k=p-1$, the $q$-expansion principle gives $E_{p-1}\in R_p$,
and \eqref{eq:Bernoulli-p-unit} gives
$(p-1)/B_{p-1}\in p\mathbb Z_{(p)}$. Consequently
\[
\mathcal S_{p-1}^\omega[x]-E_{p-1}\in pB_p.
\]
Together with the classical congruence $\overline E_{p-1}=A$
\cite[\S2.1]{katz}, this is a second proof of
\eqref{eq:Hasse-reduction}. The existence and integrality of the
section have already followed from the general differential theorem;
\eqref{eq:normalized-S-formula} separates its Eisenstein term from
the correction divisible by $p$.

The mixed-characteristic section still has poles along two-torsion.
It is its reduction on $\bar{\mathcal U}$ that extends regularly
across this locus. The same distinction applies to the pole and
ramification loci excluded in Theorem~\ref{thm:general-Hasse}.

\section{Applications to Eisenstein exactness divisors}\label{sec:applications}

The exactness criterion identifies a classical Eisenstein divisor
through higher-Schwarzian differentials. We now apply the classical
zero theorems and supersingular congruence to describe this divisor
and its reduction.

\subsection{The characteristic-zero divisor}

For $\Omega=\omega_1(\mathbb Z+\tau\mathbb Z)$ and even $k\geq4$,
\begin{equation}\label{eq:G-E}
G_k(\Omega)=2\zeta(k)\omega_1^{-k}E_k(\tau).
\end{equation}
Thus Corollary~\ref{cor:exactness-criterion} says that the order-$k$
differential of any elliptic double cover is exact exactly when
$E_k(\tau)=0$. Write uniquely
\[
k=12r+4\alpha+6\beta,\qquad r\geq0,\quad
\alpha\in\{0,1,2\},\quad\beta\in\{0,1\}.
\]
The structure of the ring of level-one modular forms gives a unique
monic polynomial $F_k(X)\in\mathbb Q[X]$ of degree $r$ such that
\begin{equation}\label{eq:divisor-polynomial}
E_k=\Delta^r E_4^\alpha E_6^\beta F_k(j),
\qquad
\Delta=\frac{E_4^3-E_6^2}{1728},\quad j=\frac{E_4^3}{\Delta}.
\end{equation}
These are the classical Eisenstein divisor polynomials
\cite[\S1]{kz-ss}. Monicity follows from the constant term $1$ of $E_k$
and the leading terms $\Delta=q+O(q^2)$ and $j=q^{-1}+O(1)$.

\begin{prop}[Classical Eisenstein consequences]\label{prop:exactness-divisor}
Away from $j=0,1728$, the order-$k$ higher-Schwarzian differential is
exact if and only if $F_k(j)=0$. The roots of $F_k$ are simple and
belong to $(0,1728)$. Representatives of these nonelliptic exactness
classes lie on the arc
\[
\{e^{i\theta}:\pi/2<\theta<2\pi/3\}.
\]
The vanishing orders at $\rho=e^{2\pi i/3}$ and $i$, measured in a
local uniformizing parameter on the moduli stack, are $\alpha$ and
$\beta$, respectively.
\end{prop}

\begin{proof}
Equation~\eqref{eq:G-E} gives the exactness criterion. Away from the
elliptic points, neither $E_4$ nor $E_6$ vanishes, while $\Delta$ has
no zero in the upper half-plane, so \eqref{eq:divisor-polynomial}
gives the polynomial criterion. The zero theorem of Rankin and
Swinnerton-Dyer \cite{rankin-swinnerton-dyer} places the remaining
zeros on the stated open arc and shows that they are simple, with
exactly the forced orders $\alpha$ and $\beta$ at its endpoints.
On the open arc $j$ is real and monotone with image $(0,1728)$.
This proves the assertions about $F_k$, including that it has no
root at either elliptic value.
\end{proof}

To specify the divisor convention, let $\nu:\mathcal M_{1,1,\mathbb C}
\to\mathbb A^1_j$ be the coarse moduli map. The exactness divisor is
the Cartier zero divisor of the section $E_k$ of the $k$th Hodge
power on this smooth moduli stack. If $\mathcal Z_{j_0}$ denotes the
reduced point substack with invariant $j_0$, it is
\[
\operatorname{div}(E_k)
=\alpha\mathcal Z_0+\beta\mathcal Z_{1728}
 +\sum_{F_k(j_0)=0}\mathcal Z_{j_0}.
\]
The local functions $j$ and $j-1728$ have orders $3$ at $\rho$ and
$2$ at $i$. Thus the rational divisor on the coarse $j$-line whose
pullback is this Cartier divisor is
\begin{equation}\label{eq:coarse-exactness-divisor}
\frac{\alpha}{3}[0]+\frac{\beta}{2}[1728]
 +\sum_{F_k(j_0)=0}[j_0].
\end{equation}
This pullback convention distinguishes the stack multiplicities from
the fractional elliptic contributions on the coarse curve.

Nozaki's theorem \cite{nozaki} gives interlacing of the nonelliptic
exactness classes in orders $k$ and $k+12$, or equivalently of the
roots of $F_k$ and $F_{k+12}$ in $(0,1728)$. Kohnen's theorem
\cite{kohnen-trans} shows that every nonelliptic zero of $E_k$ in the
standard fundamental domain is transcendental. Hence the corresponding
exactness parameters are transcendental and in particular are not CM.
These statements are applications of the classical Eisenstein zero
theorems through \eqref{eq:G-E}.

\subsection{The classical supersingular congruence}

For $p\geq5$ prime, put
\[
\delta_p=\begin{cases}1,&p\equiv2\pmod3,\\0,&p\equiv1\pmod3,\end{cases}
\qquad
\epsilon_p=\begin{cases}1,&p\equiv3\pmod4,\\0,&p\equiv1\pmod4,\end{cases}
\]
and define
\[
\operatorname{SS}_p(X)=
\prod_{\substack{j_0\in\overline{\mathbb F}_p\\
                 j_0\ \mathrm{supersingular}}}(X-j_0),
\]
with each supersingular invariant included once. The values $0$ and
$1728$ are supersingular exactly when $\delta_p=1$ and $\epsilon_p=1$,
respectively \cite[\S2]{kz-ss}.

\begin{cor}[Classical supersingular divisor congruence]
\label{cor:supersingular-polynomial}
The polynomial $F_{p-1}$ belongs to $\mathbb Z_{(p)}[X]$, and
\begin{equation}\label{eq:supersingular-polynomial}
\operatorname{SS}_p(X)
=X^{\delta_p}(X-1728)^{\epsilon_p}\overline F_{p-1}(X)
\quad\text{in }\mathbb F_p[X].
\end{equation}
In particular, $\overline F_{p-1}$ has simple roots and accounts
exactly for the nonelliptic supersingular values.
\end{cor}

\begin{proof}
This is the Eisenstein-series case of \cite[Theorem~1]{kz-ss}, in the
normalization
\[
E_{p-1}=\Delta^r E_4^{\delta_p}E_6^{\epsilon_p}F_{p-1}(j).
\]
That theorem gives $p$-integrality and the complete polynomial
congruence. Since both polynomials are monic, the constant factor is
$1$. Squarefreeness of $\operatorname{SS}_p$ gives the last assertion.
\end{proof}

The congruence is the classical result of Deligne; see also
\cite{gomez-lakein-larsen}. Its role here is to describe the reduction
of the Eisenstein exactness divisor. The equality of sections in
Corollary~\ref{cor:Hasse-reduction} is compatible with this polynomial
identity, but is not a substitute for its multiplicity statement.

\subsection{The weight-twelve example}

The first nonelliptic exactness class occurs in weight twelve. Since
\[
E_{12}=E_4^3-\frac{432000}{691}\Delta,
\qquad F_{12}(X)=X-\frac{432000}{691},
\]
its invariant is $j(\tau)=432000/691$, and its representative $\tau$
on the open unit-circle arc is transcendental by Kohnen's theorem.
In characteristic $13$ the same divisor polynomial satisfies
\[
F_{12}(X)\equiv X+8\pmod{13}.
\]
Here $\delta_{13}=\epsilon_{13}=0$, so this is the supersingular
polynomial itself. It exhibits the characteristic-zero exactness
class and its supersingular reduction in the same normalization.

\raggedbottom

\end{document}